\documentclass{amsart}
\usepackage{amssymb}
\usepackage{float}
\usepackage{array}
\usepackage[margin=1in]{geometry}
\makeatletter
\newcommand{\thickhline}{%
    \noalign {\ifnum 0=`}\fi \hrule height 1pt
    \futurelet \reserved@a \@xhline
}
\newcolumntype{"}{@{\hskip\tabcolsep\vrule width 1pt\hskip\tabcolsep}}
\makeatother

\newtheorem{theorem}{Theorem}[section]
\newtheorem{lemma}[theorem]{Lemma}

\theoremstyle{definition}

\newtheorem{proposition}[theorem]{Proposition}

\theoremstyle{remark}

\numberwithin{equation}{section}

\def\m{{\mathfrak m}}

\def\bb{\beta}
\newcommand{\rank}{\operatorname{rank}}

\begin{document}
\title{Betti numbers for certain Cohen-Macaulay tangent cones}
\author{Mesut \c{S}ah\.{i}n}
\address{Department of Mathematics,
Hacettepe University, Ankara,  06800 Turkey}
\email{mesut.sahin@hacettepe.edu.tr}
\author{N\.{i}l \c{S}ah\.{i}n}
\address{Department of Industrial Engineering, Bilkent University, Ankara, 06800 Turkey}
\email{nilsahin@bilkent.edu.tr}

\thanks{The authors were supported by the project 114F094 under the program 1001 of the
Scientific and Technological Research Council of Turkey.}


\date{\today}

\commby{}

\dedicatory{}

\begin{abstract}
In this article, we compute Betti numbers for a Cohen-Macaulay tangent cone of a monomial curve in the affine $4$-space corresponding to a pseudo symmetric numerical semigroup. As a byproduct, we also show that for these semigroups, being of homogeneous type and homogeneous are equivalent properties. 
\end{abstract}

\maketitle

\section{introduction}
Let $S=\langle n_1,\dots,n_k \rangle=\{ \displaystyle u_1n_1+\cdots+u_kn_k | u_i \in \mathbb{N}\}$ be a numerical semigroup generated by the positive integers $n_1,\dots,n_k$ with $\gcd (n_1,\dots,n_k)=1$. $K$ being a field, let $A=K[X_1,X_2,\dots,X_k]$ and $K[S]$ be the semigroup ring $K[t^{n_1}, t^{n_2}, \dots, t^{n_k}]$ of $S$. Then $K[S]\simeq A/I_S$ where $I_S$ is the kernel of the surjection $A \stackrel{\phi_0}{\longrightarrow} K[S]$, associating $X_i$ to $t^{n_i}$.  If $C_S$ is the affine curve with parameterization 
$$X_1=t^{n_1},\ \ X_2=t^{n_2},\ \dots,\  X_k=t^{n_k}  $$
corresponding to $S$, and $1\notin S$ then the curve is singular at the origin. The smallest minimal generator of $S$ is called the \textit{multiplicity} of $C_S$. In order to understand this singularity deeper, it is natural to study algebraic properties of the local ring $R_S=K[[t^{n_1},\dots,t^{n_k}]]$ with the maximal ideal $\m=\langle t^{n_1},\dots,t^{n_k}\rangle$ and its associated graded ring 
$$gr_{\mathfrak{m}}(R_S)=\bigoplus_{i=0}^{\infty} \mathfrak{m}^i/\mathfrak{m}^{i+1}\cong A/{I^*_S},$$
where ${I^*_S}=\langle f^*| f \in I_S \rangle$ with $f^*$ denoting the least homogeneous summand of $f$. When $K$ is algebraically closed, $K[S]$ is the coordinate ring of the monomial curve $C_S$ and $gr_{\mathfrak{m}}(R_S)$ is the coordinate ring of its tangent cone. One natural set of invariants for these coordinate rings attracting attention is the Betti sequence. We refer to the nice survey \cite{stamate} written by Stamate for a comprehensive literature on this subject. Recall that the Betti sequence $\beta(M)=(\bb_0,\dots,\bb_{k-1})$ of an $A$ module $M$ is the sequence consisting of the ranks of the free modules in a minimal free resolution ${\bf F}$ of $M$, where
$${\bf F}: \ 0\longrightarrow A^{\beta_{k-1}}\stackrel{}{\longrightarrow}\cdots\stackrel{}{\longrightarrow}A^{\beta_1}\stackrel{}{\longrightarrow}A^{\beta_0}.$$

When $\beta(A/{I^*_S})=\beta(K[S])$ the semigroup $S$ is said to be of homogeneous type as defined by \cite{HerzogRossiValla}. In particular, if a semigroup is of homogeneous type then the Betti sequence of its Cohen-Macaulay tangent cone can be obtained from a minimal free resolution of $K[S]$. So as to take advantage of this idea, Jafari and Zarzuela  introduced the concept of \textit{homogeneous} semigroup in \cite{jafari}. When the multiplicity of a monomial curve corresponding to a homogeneous semigroup is $n_i$, homogeneity guarantees the existence of a minimal generating set for $I_S$ whose image under the map $\pi_i : A \rightarrow \bar{A}=K[X_1,\dots,\bar{X_i},\dots,X_k]$ is homogeneous, where $\pi(X_i) = \bar{X_i}=0$ and $\pi(X_j) = X_j$ for $i \neq j$. Together with the assumption to have a Cohen-Macaulay tangent cone, this property is inherited by a standard basis of $I_S$ and the authors of \cite{jafari} were able to prove that $S$ is of homogeneous type. The converse is not true in general: there exists a $3$-generated numerical semigroup with a complete intersection tangent cone which is of homogeneous type but not homogeneous, see \cite[Example 3.19]{jafari}. They also ask in \cite[Question 4.22]{jafari} if there are $4$-generated semigroups of homogeneous type which are not homogeneous having non complete intersection tangent cones. As homogeneous type semigroups have Cohen-Macaulay tangent cones, we restrict our attention to monomial curves having Cohen-Macaulay tangent cones in this article.

The problem of determining the Betti sequence for the tangent cone (see \cite[Problem 9.9]{stamate}) is studied for 4 generated symmetric monomial curves by Mete and Zengin \cite{mz}. In this paper, we focus on the next interesting case of 4-generated pseudo symmetric monomial curves. Using the standard bases we obtained in \cite{SahinSahin}, we determine the Betti sequence for the tangent cone addressing \cite[Problem 9.9]{stamate}) for 4-generated pseudo symmetric monomial curves having Cohen-Macaulay tangent cones, and prove that being homogeneous and of homogeneous type is equivalent answering \cite[Question 4.22]{jafari}. So, in most cases, there is no $4$ generated pseudo symmetric numerical semigroup of homogeneous type which is not homogeneous. Before we state our main result, let us recall from \cite{komeda} that a $4$-generated semigroup $S=\langle n_1,n_2,n_3,n_4 \rangle$ is pseudo-symmetric if and only if there are integers $\alpha_i>1$, for
$1\le i\le4$, and $\alpha_{21}>0$ with $\alpha_{21}<\alpha_1-1$,
such that 
\begin{eqnarray*}
n_1&=&\alpha_2\alpha_3(\alpha_4-1)+1,\\
n_2&=&\alpha_{21}\alpha_3\alpha_4+(\alpha_1-\alpha_{21}-1)(\alpha_3-1)+\alpha_3,\\
n_3&=&\alpha_1\alpha_4+(\alpha_1-\alpha_{21}-1)(\alpha_2-1)(\alpha_4-1)-\alpha_4+1,\\
n_4&=&\alpha_1\alpha_2(\alpha_3-1)+\alpha_{21}(\alpha_2-1)+\alpha_2. 
\end{eqnarray*}
Then, the toric ideal $I_S=\langle f_1,f_2,f_3,f_4,f_5 \rangle$ with
\begin{eqnarray*} f_1&=&X_1^{\alpha_1}-X_3X_4^{\alpha_4-1},  \quad \quad
f_2=X_2^{\alpha_2}-X_1^{\alpha_{21}}X_4, \quad
f_3=X_3^{\alpha_3}-X_1^{\alpha_1-\alpha_{21}-1}X_2,\\
f_4&=&X_4^{\alpha_4}-X_1X_2^{\alpha_2-1}X_3^{\alpha_3-1}, \quad 
f_5=X_1^{\alpha_{21}+1}X_3^{\alpha_3-1}-X_2X_4^{\alpha_4-1}.
\end{eqnarray*}  
 The Betti sequence of $K[S]$ for a $4$-generated pseudo symmetric semigroup is $\bb(K[S])=(1,5,6,2)$ by \cite{barucci}. Hence, $S$ is of homogeneous type if and only if the Betti sequence of the tangent cone is also $\bb(A/{I^*_S})=(1,5,6,2)$. We refer the reader to \cite{eto} for the Betti sequence of $K[S]$ for $4$-generated almost symmetric semigroups.

Our main result is as follows:

\begin{theorem} \label{main}Let $S$ be a $4$-generated pseudo-symmetric semigroup with a Cohen-Macaulay tangent cone. Then, the Betti sequence $\beta(A/{I^*_S})$ of the tangent cone is

\begin{itemize}
\item  $\beta(A/{I^*_S})=(1,5,6,2)$ if $n_1$ is the multiplicity, 
\item  $\beta(A/{I^*_S})=(1,5,6,2)$ if $n_2$ is the multiplicity and $\alpha_1=\alpha_4$,\\
$\beta(A/{I^*_S})=(1,5,7,3)$ if $n_2$ is the multiplicity and $\alpha_1<\alpha_4$,\\
$\beta(A/{I^*_S})=(1,6,9,4)$ if $n_2$ is the multiplicity and $\alpha_1>\alpha_4$,
\item  $\beta(A/{I^*_S})=(1,5,6,2)$ if $n_3$ is the multiplicity,  and $\alpha_2=\alpha_{21}+1$,\\
$\beta(A/{I^*_S})=(1,6,8,3)$ if $n_3$ is the multiplicity,  and $\alpha_2<\alpha_{21}+1$,
\item  $\beta(A/{I^*_S})=(1,5,6,2)$ if $n_4$ is the multiplicity, and $\alpha_3=\alpha_1-\alpha_{21}$,\\
 $\beta(A/{I^*_S})=(1,5,7,3)$ if $n_4$ is the multiplicity, and $\alpha_3<\alpha_1-\alpha_{21}$.
\end{itemize}

\end{theorem}

 We illustrate in Table \ref{tab:1} that there are pseudo symmetric monomial curves with Cohen-Macaulay tangent cones in all of these cases.

\begin{table}
\caption{Examples of each case}\label{tab:1}
\centering
\begin{tabular}{|c|c|c|c|c||c|c|c|c||c|c|c|c|}
  \hline
  $\alpha_{21}$&$\alpha_1$& $\alpha_2$ & $\alpha_3$& $\alpha_4$&$n_1$&$n_2$&$n_3$&$n_4$&$\beta_0$&$\beta_1$&$\beta_2$&$\beta_3$\\ 
  \hline
  2 &5 &3 &2 &2 &7 &12 & 13 & 22 &1 &5 & 6 & 2\\
 \hline
 2&4&4&2&4&25&19&22&26&1 &5 & 6 & 2\\
  \hline
 2&4&4&2&5&33&23&28&26&1 &5 & 7 & 3\\
 \hline
  2&5&4&2&4&25&20&35&30&1 &6 & 9 & 4\\
  \hline
 1&3&2&3&3&13&14&9&15&1 &5 & 6 & 2\\
  \hline
 3&6&3&4&6&61&82&51&63&1 &6 & 8 & 3\\
  \hline
   1&3&2&2&4&13&11&12&9&1 &5 & 6 & 2\\
    \hline
 1&4&2&2&4&13&12&19&11&1 &5 &7 & 3 \\
  \hline
\end{tabular}
\end{table}

We make repeated use of the following effective result as in \cite{HS,jafari,stamate} in order to reduce the number of cases for determining the Betti numbers of the tangent cones. 
\begin{lemma}\label{HS} Assume that the multiplicity of the monomial curve $C_S$ is $n_i$. Let the $K$-algebra homomorphism $\pi_i : A \rightarrow \bar{A}=K[X_1,\dots,\bar{X_i},\dots,X_k]$ be defined by $\pi_i(X_i) = \bar{X_i}=0$ and $\pi_i(X_j) = X_j$ for $i \neq j$, and set $\bar{I}=\pi_i(I^*_S)$. If the tangent cone $gr_{\m}(R_S)$ is Cohen-Macaulay, then the Betti sequences of $gr_{\m}(R_S)$ and that of $\bar{A}/\bar{I}$ are the same.
\end{lemma}
\begin{proof} If the tangent cone $gr_{\m}(R_S)$ is Cohen-Macaulay, then $X_i$ is regular on $A/{I^*_S}$. The result follows from the well known fact that Betti sequences are the same up to a regular sequence.
\end{proof}

Therefore, the problem of determining the Betti sequence of the tangent cone is reduced to computing the Betti sequence of the ring $\bar{A}/\bar{I}$. In all proofs about the minimal free resolution of $\bar{A}/\bar{I}$ we use the following criterion by Buchsbaum-Eisenbud to confirm the exactness, leaving the not so difficult task of checking if it is a complex to the reader.
\begin{theorem}\cite[Corollary 2]{bu-ei}
Let
$$0{\longrightarrow}F_{k-1}\stackrel{\phi_{k-1}}{\longrightarrow}\cdots\stackrel{\phi_2}{\longrightarrow}F_1\stackrel{\phi_1}{\longrightarrow}F_0$$
be a complex of free modules over a Noetherian ring $A$. Let {\rm rank}$({\phi_i})$ be the size of the largest nonzero minor of the matrix describing $\phi_i$, and let $I(\phi_i)$ be the
ideal generated by the minors of maximal rank. Then the complex is exact if and only if 

(a) $\mbox {\rm rank}(\phi_{i+1})+\mbox {\rm rank}(\phi_i)=\mbox {\rm rank}(F_i)$ and

(b) $I(\phi_i)$ contains an $A$-sequence of length $i$

\noindent
for all $1\le i\le k-1$.
\end{theorem}

The structure of the paper is as follows: We treat the cases where $S$ is homogeneous in the next section and, when $S$ is not homogeneous we find the minimal free resolution of the ring $\bar{A}/\bar{I}$ in each subsequent section, completing the proof of Theorem \ref{main} by the virtue of Lemma \ref{HS}. We refer the reader to the book \cite{greuel-pfister} for the basics of Commutative Algebra as we use Singular \cite{singular} in our computations.
\section{Homogeneous cases}

In this section, we characterize which pseudo symmetric $4$-generated semigroups are homogeneous. We start recalling basic definitions from \cite{jafari}. The Ap\'ery set of $S$ with respect to $s\in S$ is defined to be $AP(S,s)=\{x\in S\ |\ x-s \notin S\}$ and the set of lengths of $s$ in $S$ is $$L(s)=\left\lbrace \displaystyle\sum\limits_{i=1}^{k}u_i \quad | \quad s=\sum\limits_{i=1}^{k}u_in_i, u_i\geq 0\right\rbrace.$$ Note that $L(s)$ is the set of standard degrees of monomials $X_1^{u_1}\cdots X_k^{u_k}$ of $S$-degree $\deg_S(X_1^{u_1}\cdots X_k^{u_k})=s$. A subset $T\subset S$ is said to be homogeneous if either it is empty or $L(s)$ is a singleton for all $0\neq s \in T$. $n_i$ being the smallest among $n_1, n_2, \dots, n_k$, the semigroup $S$ is said to be \textit{homogeneous} if the Ap\'ery set $AP(S,n_i)$ is homogeneous. 

\begin{proposition}\label{prop1}
Let $S$ be a 4 generated pseudo symmetric numerical semigroup. Then $S$ is homogeneous if and only if
\begin{itemize}
\item  $n_1$ is the multiplicity or
\item  $n_2$ is the multiplicity and $\alpha_1=\alpha_4$ or
\item  $n_3$ is the multiplicity and $\alpha_2=\alpha_{21}+1$ or
\item  $n_4$ is the multiplicity and $\alpha_3=\alpha_1-\alpha_{21}$.
\end{itemize}
\end{proposition}
\begin{proof}
By Corollary 3.10 of \cite{jafari}, $S$ is homogeneous if and only if there exists a set $E$ of minimal generators for $I_S$ such that every nonhomogeneous element of $E$ has a term that is divisible by $X_i$ when $n_i$ is the multiplicity. Corollary 2.4 of \cite{SahinSahin} states that indispensable binomials of $I_S$ are $\{f_1, f_2, f_3, f_4,f_5 \}$ if $\alpha_1-\alpha_{21}>2$ and are $\{f_1, f_2, f_3, f_5\}$ if $\alpha_1-\alpha_{21}=2$. Therefore, they must appear in every minimal generating set. Let us take $E=\{f_1,\hdots, f_5\}$ in order to prove sufficiency of the conditions. 
\begin{itemize}
\item Since each $f_j$ ($j=1,\hdots, 5$) has a term that is divisible by $X_1$, when $n_1$ is the multiplicity, $S$ is always homogeneous.
\item The only binomial in $E$ that has no monomial term divisible by $X_2$ is $f_1$. Hence when $n_2$ is the multiplicity and $\alpha_1=\alpha_4$, it follows that $f_1$ and thus $S$ is homogeneous.
\item The only binomial in $E$ that has no monomial term divisible by $X_3$ is $f_2$. Hence when $n_3$ is the multiplicity and $\alpha_2=\alpha_{21}+1$, $f_2$ and thus $S$ is homogeneous.
\item Similarly, only $f_3$ has no monomial term that is divisible by $X_4$ and it is homogeneous when $\alpha_3=\alpha_1-\alpha_{21}$. Hence, $S$ is homogeneous if $n_4$ is the multiplicity. 
\end{itemize} 
For the necessity of these conditions, recall that $f_1,f_2$ and $f_3$ are indispensable, so they must be homogeneous when the multiplicity is $n_2,n_3$ and $n_4$, respectively.  
\end{proof}
\section{The proof when the multiplicity is $n_1$ }\label{resolutions}
If the tangent cone is Cohen-Macaulay and the semigroup is homogeneous, it is known that the semigroup is of homogeneous type. When $n_1$ is the multiplicity, the pseudo symmetric semigroup is always homogeneous by Proposition \ref{prop1} and hence the Betti sequence is $(1,5,6,2)$ in this case. 
\section{The proof when the multiplicity is $n_2$}
Let $n_2$ be the multiplicity and the tangent cone is Cohen-Macaulay. If $ \alpha_1 = \alpha_4$, then the Betti sequence is $(1,5,6,2)$ by Proposition \ref{prop1}. So, we treat the cases $\alpha_1 < \alpha_4$ and $\alpha_1 > \alpha_4$ separately.
\subsection{The proof in the case $\alpha_1 < \alpha_4$} 
In this case, $\{f_1,f_2,f_3,f_4,f_5\}$ is a standard basis of $I_S$ by \cite[Lemma 3.8]{SahinSahin}. Since, ${\bar{I}}$ is the image of $I^*_S$ under the map $\pi_2$ sending only $X_2$ to $0$, it follows that ${\bar{I}}$ is generated by 	
$G_*=\{X_1^{\alpha_1},X_1^{\alpha_{21}}X_4, X_3^{\alpha_3}, X_4^{\alpha_4}, X_1^{\alpha_{21}+1}X_3^{\alpha_3-1}\}.$
 We prove the claim by demonstrating that the following complex is a minimal free resolution of $\bar{A}/\bar{I}$ by the virtue of Lemma \ref{HS}:
$$0\longrightarrow A^3\stackrel{\phi_3}{\longrightarrow}A^7\stackrel{\phi_2}{\longrightarrow}A^5\stackrel{\phi_1}{\longrightarrow}A\longrightarrow 0$$
	where $\phi_1=\begin{bmatrix} 
	X_1^{\alpha_1}& X_1^{\alpha_{21}}X_4& X_3^{\alpha_3}& X_4^{\alpha_4}& X_1^{\alpha_{21}+1}X_3^{\alpha_3-1}
	\end{bmatrix}$\\

$\phi_2=\begin{bmatrix}
 0& X_4&0 &0 &X_3^{\alpha_3-1} & 0& 0  \\
 0& -X_1^{\alpha_1-\alpha_{21}}&X_1X_3^{\alpha_3-1} &X_4^{\alpha_4-1} & 0&0 &-  X_3^{\alpha_3} \\
 X_1^{\alpha_{21}+1}& 0& 0&0 &0 &X_4^{\alpha_4} & X_1^{\alpha_{21}}X_4 \\
 0& 0& 0&-X_1^{\alpha_{21}} & 0&-X_3^{\alpha_3} & 0 \\
 -X_3& 0& -X_4& 0& -X_1^{\alpha_1-\alpha_{21}-1} &0 &0  
\end{bmatrix}$ \\

$\phi_3=\begin{bmatrix}
 -X_4& 0& 0 \\
 0& X_3^{\alpha_3-1}& 0 \\
 X_3 &X_1^{\alpha_1-\alpha_{21}-1} & 0\\
 0& 0&-X_3^{\alpha_3}  \\
 0 & -X_4& 0\\
 0 &0 &X_1^{\alpha_{21}} \\
 X_1 & 0&-X_4^{\alpha_4-1} 
\end{bmatrix}$ 

It is easy to check that $\rank \phi_1=1$, $\rank \phi_2=4$, $\rank \phi_3=3$. So, we show that $I(\phi_i)$ contains a regular sequence of length $i$, for all $i=1,2,3$. Since this is obvious for $i=1$, we discuss the other cases. For the matrix $\phi_2$, the $4$-minor corresponding to the rows $1,2,4,5$ and columns $1,5,6,7$ is computed to be $-X_3^{3\alpha_3}$. Similarly, the $4$-minor corresponding to the rows $2,3,4,5$ and columns $1,2,4,5$ is $X_1^{2\alpha_1}$. As these minors are relatively prime, the ideal $I(\phi_2)$ contains a regular sequence of length $2$. The $3$-minors of $\phi_3$ corresponding to the rows $1,5,7$ is $-X_4^{1+\alpha_4}$, to the rows $2,3,4$ is $X_3^{2\alpha_3}$ and to the rows $3,6,7$ is $X_1^{\alpha_1}$. As they are powers of different variables, they constitute a regular sequence of length $3$.
\subsection{The proof in the case $\alpha_1 > \alpha_4$} 
In this case, a standard basis of $I_S$ is, by \cite[Lemma 3.8]{SahinSahin}, the set
$\{f_1,f_2,f_3,f_4,f_5,f_6=X_1^{\alpha_1+\alpha_{21}}-X_2^{\alpha_2}X_3X_4^{\alpha_4-2}\}$. Since, ${\bar{I}}$ is the image of $I^*_S$ under the map $\pi_2$ sending only $X_2$ to $0$, it follows that ${\bar{I}}$ generated by 	
$$G_*=\{X_3X_4^{\alpha_4-1},X_1^{\alpha_{21}}X_4,X_3^{\alpha_3},X_4^{\alpha_4},X_1^{\alpha_{21}+1}X_3^{\alpha_3-1},X_1^{\alpha_1+\alpha_{21}}\}.$$
 We prove the claim by demonstrating that the following complex is a minimal free resolution of $\bar{A}/\bar{I}$ by the virtue of Lemma \ref{HS}:
$$0\longrightarrow A^4\stackrel{\phi_3}{\longrightarrow}A^9\stackrel{\phi_2}{\longrightarrow}A^6\stackrel{\phi_1}{\longrightarrow}A\longrightarrow 0$$
where 
$$\phi_1=\begin{bmatrix} 
X_3X_4^{\alpha_4-1}&X_1^{\alpha_{21}}X_4&X_3^{\alpha_3}&X_4^{\alpha_4}&X_1^{\alpha_{21}+1}X_3^{\alpha_3-1}&X_1^{\alpha_1+\alpha_{21}}
\end{bmatrix}$$ with $\phi_2$ equals to:\\
	
$\begin{bmatrix}
 -X_4& 0& 0& 0& 0&X_1^{\alpha_{21}} & 0&X_3^{\alpha_3-1} & 0 \\
 0& 0&-X_1^{\alpha_1} & -X_1X_3^{\alpha_3-1} &-X_4^{\alpha_4-1} & -X_3X_4^{\alpha_4-2}& 0& 0 &  X_3^{\alpha_3} \\
  0& -X_1^{\alpha_{21}+1}& 0 & 0 &0 & 0& 0 &-X_4^{\alpha_4-1}& -X_1^{\alpha_{21}}X_4 \\

 X_3& 0 & 0 & 0& X_1^{\alpha_{21}} & 0& 0 &0 & 0 \\
  0& X_3& 0& X_4&0 &0 &-X_1^{\alpha_1-1} & 0& 0 \\
 0& 0&  X_4& 0& 0& 0 & X_3^{\alpha_3 -1}& 0& 0
\end{bmatrix}$ \\

and\\

$\phi_3=\begin{bmatrix}
0& -X_1^{\alpha_{21}}& 0& 0\\
X_4& 0& 0 & 0\\
0& 0& -X_3^{\alpha_3-1}& 0 \\
X_3& 0 &X_1^{\alpha_1-1} & 0\\
 0&X_3 & 0& 0 \\
 0 & -X_4& 0& -X_3^{\alpha_3-1}\\
 0& 0& X_4& 0 \\
 0 &0 &0& X_1^{\alpha_{21}} \\
 X_1 & 0& 0 &-X_4^{\alpha_4-2} 
\end{bmatrix}$ 

It is easy to check that $\rank \phi_1=1$, $\rank \phi_2=5$, $\rank \phi_3=4$. So, we show that $I(\phi_i)$ contains a regular sequence of length $i$, for all $i=1,2,3$. Since this is obvious for $i=1$, we discuss the other cases. For the matrix $\phi_2$, the $5$-minor corresponding to the rows $1, 2, 3, 5, 6$ and columns $1, 3, 4, 5, 8$ is computed to be $-X_4^{1+2\alpha_4}$. Similarly, the $5$-minor corresponding to the rows $1,2, 4, 5, 6$ and columns $1, 2, 7, 8, 9$ is $-X_3^{3\alpha_3}$. As these minors are powers of different variables, the ideal $I(\phi_2)$ contains a regular sequence of length $2$. The $4$-minors of $\phi_3$ corresponding to the rows $1, 4, 8, 9$ is $X_1^{2\alpha_{21}+\alpha_1}$, to the rows $3, 4, 5, 6$ is $X_3^{2\alpha_3}$ and to the rows $2, 6, 7, 9$ is $-X_4^{1+\alpha_4}$. As they are powers of different variables they constitute a regular sequence of length $3$.

\section{The proof when the multiplicity is $n_3$}

Suppose that the tangent cone is Cohen-Macaulay. If $\alpha_2 = \alpha_{21}+1$, then the Betti sequence is $(1,5,6,2)$ by Proposition \ref{prop1}. If $\alpha_2 < \alpha_{21}+1$, then a minimal standard basis for $I_S$ is either $\{f_1,f_2, f_3, f_4, f_5, f_6=X_1^{\alpha_1-1}X_4-X_2^{\alpha_2-1}X_3^{\alpha_3}\}$ or $\{f_1,f_2,f_3,f_4'=X_4^{\alpha_4}-X_2^{\alpha_2-2}X_3^{2\alpha_3-1},f_5,f_6\}$, by \cite[Lemma 3.12]{SahinSahin}.  Since, $\pi_3$ sends only $X_3$ to $0$, it follows that in both cases the ideal ${\bar{I}}=\pi_3(I^*_S)$ is generated by 	
	$$G_*=\{X_1^{\alpha_1},X_2^{\alpha_2}, X_1^{\alpha_1-\alpha_{21}-1}X_2, X_4^{\alpha_4},X_2X_4^{\alpha_4-1}, X_1^{\alpha_1-1}X_4\}.$$
 We prove the claim by demonstrating that the following complex is a minimal free resolution of $\bar{A}/\bar{I}$ by the virtue of Lemma \ref{HS}:
	$$0\longrightarrow A^3\stackrel{\phi_3}{\longrightarrow}A^8\stackrel{\phi_2}{\longrightarrow}A^6\stackrel{\phi_1}{\longrightarrow}A\longrightarrow 0$$
	where $\phi_1=\begin{bmatrix} 
X_1^{\alpha_1}& X_2^{\alpha_2}&  X_1^{\alpha_1-\alpha_{21}-1}X_2& X_4^{\alpha_4}& X_2X_4^{\alpha_4-1}& X_1^{\alpha_1-1}X_4
	\end{bmatrix}$\\

$\phi_2=\begin{bmatrix}
 0& -X_4&0& 0&0 &0 &X_2 & 0 \\
 0&0& X_1^{\alpha_1-\alpha_{21}-1}&0& -X_4^{\alpha_4-1} & 0&0 & 0 \\
- X_4^{\alpha_4-1}&0 &-X_2^{\alpha_2-1} &0&0 & -X_1^{\alpha_{21}}X_4&-X_1^{\alpha_{21}+1} & 0  \\
0& 0& 0& X_2& 0& 0& 0& X_1^{\alpha_1-1} \\
 X_1^{\alpha_1-\alpha_{21}-1}&0&0& -X_4&X_2^{\alpha_2-1}& 0 & 0 &0 \\
 
  0&X_1&0& 0& 0&X_2& 0& -X_4^{\alpha_4-1} 
\end{bmatrix}$ \\

$\phi_3=\begin{bmatrix}
 0& -X_2^{\alpha_2-1}& -X_1^{\alpha_{21}}X_4 \\
 -X_2& 0& 0 \\
 0& X_4^{\alpha_4-1}& 0\\
 0& 0& -X_1^{\alpha_1-1}\\
 0 &X_1^{\alpha_1-\alpha_{21}-1} & 0\\
  X_1 & 0&X_4^{\alpha_4-1} \\
 -X_4& 0 & 0\\
 0 &0 &X_2 \\

\end{bmatrix}$ 

It is easy to check that $\rank \phi_1=1$, $\rank \phi_2=5$, $\rank \phi_3=3$. So, we show that $I(\phi_i)$ contains a regular sequence of length $i$, for all $i=1,2,3$. Since this is obvious for $i=1$, we discuss the other cases. For the matrix $\phi_2$, the $5$-minor corresponding to the rows $1, 2, 3, 5, 6$ and columns $1, 2, 4, 5, 8$ is computed to be $-X_4^{3\alpha_4-1}$. Similarly, the $5$-minor corresponding to the rows $2, 3, 4, 5, 6$ and columns $1, 2, 3, 7, 8$ is $-X_1^{3\alpha_1-\alpha_{21}-1}$. As these minors are powers of different variables, the ideal $I(\phi_2)$ contains a regular sequence of length $2$. The $3$-minors of $\phi_3$ corresponding to the rows $1, 2, 8$ is $-X_2^{\alpha_{2}+1}$, to the rows $3, 6, 7$ is $-X_4^{2\alpha_4-1}$ and to the rows $4, 5, 6$ is $X_1^{2\alpha_1-\alpha_{21}-1}$. As they are powers of different variables, they constitute a regular sequence of length $3$.

\section{The proof when the multiplicity is $n_4$}
Suppose that the tangent cone is Cohen-Macaulay. If $\alpha_3 = \alpha_{1}-\alpha_{21}$, then the Betti sequence is $(1,5,6,2)$ by Proposition \ref{prop1}.  If $\alpha_3 < \alpha_{1}-\alpha_{21}$, then a minimal standard basis for $I_S$ is $\{f_1,f_2, f_3, f_4, f_5\}$, by \cite[Lemma 3.17]{SahinSahin}.  Since ${\bar{I}}=\pi_4(I^*_S)$, under the map $\pi_4$ sending only $X_4$ to $0$, it is generated by 	
	$$G_*=\{X_1^{\alpha_1}, X_2^{\alpha_2}, X_3^{\alpha_3}, X_1X_2^{\alpha_2-1}X_3^{\alpha_3-1},X_1^{\alpha_{21}+1}X_3^{\alpha_3-1} \}.$$
 We prove the claim by demonstrating that the following complex is a minimal free resolution of $\bar{A}/\bar{I}$ by the virtue of Lemma \ref{HS}:
$$0\longrightarrow A^3\stackrel{\phi_3}{\longrightarrow}A^7\stackrel{\phi_2}{\longrightarrow}A^5\stackrel{\phi_1}{\longrightarrow}A\longrightarrow 0$$
	where $\phi_1=\begin{bmatrix} 
	X_1^{\alpha_1}& X_2^{\alpha_2}& X_3^{\alpha_3}& X_1X_2^{\alpha_2-1}X_3^{\alpha_3-1}& X_1^{\alpha_{21}+1}X_3^{\alpha_3-1}
	\end{bmatrix}$\\

$\phi_2=\begin{bmatrix}
0 & X_2^{\alpha_2}& 0& 0& X_3^{\alpha_3-1}&0 & 0 \\
0 & -X_1^{\alpha_1}& -X_1X_3^{\alpha_3-1}& 0& 0&0 & -X_3^{\alpha_3} \\
-X_1^{\alpha_{21}+1} &0 & 0& 0& 0&-X_1X_2^{\alpha_2-1} & X_2^{\alpha_2} \\
0 & 0& X_2& -X_1^{\alpha_{21}}&0 & X_3& 0 \\
X_3 & 0& 0& X_2^{\alpha_2-1}& -X_1^{\alpha_1-\alpha_{21}-1}& 0& 0 
\end{bmatrix}$ \\

$\phi_3=\begin{bmatrix}
0 & -X_2^{\alpha_2-1}& 0 \\
0 & 0& -X_3^{\alpha_3-1} \\
-X_3 & 0& X_1^{\alpha_1-1} \\
0 & X_3& X_1^{\alpha_1-\alpha_{21}-1}X_2 \\
0 & 0& X_2^{\alpha_2} \\
X_2 & X_1^{\alpha_{21}}& 0\\
X_1& 0& 0  
\end{bmatrix}.$

It is easy to check that $\rank \phi_1=1$, $\rank \phi_2=4$, $\rank \phi_3=3$. So, we show that $I(\phi_i)$ contains a regular sequence of length $i$, for all $i=1,2,3$. Since this is obvious for $i=1$, we discuss the other cases. For the matrix $\phi_2$, the $4$-minor corresponding to the rows $1,3,4,5$ and columns $2,3,4,7$ is computed to be $X_2^{3\alpha_2}$. Similarly, the $4$-minor corresponding to the rows $2,3,4,5$ and columns $1,2,4,5$ is $-X_1^{2\alpha_1+\alpha_{21}}$. As these minors are relatively prime, the ideal $I(\phi_2)$ contains a regular sequence of length $2$. The $3$-minors of $\phi_3$ corresponding to the rows $1,5,6$ is $-X_2^{2\alpha_2}$, to the rows $2,3,4$ is $X_3^{1+\alpha_3}$ and to the rows $3,6,7$ is $-X_1^{\alpha_1+\alpha_{21}}$. As they are powers of different variables they constitute a regular sequence of length $3$.

\section*{Acknowledgements}

The authors thank the anonymous referee for comments improving the presentation of the paper.


\begin{thebibliography}{99}

\bibitem{barucci} V. Barucci, R. Fr\"{o}berg and M. \c{S}ahin, \textit{On free resolutions of some semigroup rings}, J. Pure Appl. Algebra 218(2014), no. 6,1107-1116.

\bibitem{bu-ei} D.~Buchsbaum, D.~Eisenbud, {\it What makes a complex exact?}, J. Algebra {\bf 25} (1973), 259--268.

\bibitem{eto} K. Eto, {\it Almost Gorenstein monomial curves in affine four space}, J. Algebra 488 (2017), 362-387.

\bibitem{greuel-pfister} G-M Greuel, G. Pfister, \textit{A Singular Introduction to Commutative
Algebra}, Springer-Verlag, 2002.


\bibitem{singular} G.-M. Greuel, G. Pfister, and H. Sch\"{o}nemann.
{\sc Singular} 2.0. A Computer Algebra System for Polynomial
Computations. Centre for Computer Algebra, University of
Kaiserslautern (2001). {\tt http://www.singular.uni-kl.de}.


\bibitem{HS} J. Herzog, D. I. Stamate, {\it On the defining equations of the tangent cone of a numerical semigroup ring}, J. Algebra 418 (2014), 8-28. 


\bibitem{HerzogRossiValla} J. Herzog, M. E. Rossi, and G. Valla, \textit{On the depth of the symmetric algebra}, Trans. Amer. Math.
Soc. 296 (2) (1986), 577-606.

\bibitem{jafari} R. Jafari, S. Zarzuela Armengou, \textit{Homogeneous numerical semigroups}, Semigroup Forum (2018). https://doi.org/10.1007/s00233-018-9941-6

\bibitem{komeda} J. Komeda, \textit{On the existence of Weierstrass points with a certain semigroup}, Tsukuba J. Math. 6 (1982), no. 2, 237-270.

\bibitem{mz} P. Mete and E. E. Zengin, {\it Minimal Free Resolutions of the Tangent Cones of Gorenstein Monomial Curves}, arXiv:1801.04956 

\bibitem{SahinSahin} M. \c{S}ahin and N. \c{S}ahin, \textit{On pseudo symmetric monomial curves}, Comm. Algebra (2018), 46:6, 2561-2573.

\bibitem{stamate} Stamate D.I (2018).  Betti Numbers for Numerical Semigroup Rings. In: Ene V., Miller E. (eds) Multigraded Algebra and Applications. NSA 2016. Springer Proceedings in Mathematics and Statistics, vol 238. Springer, Cham



\end{thebibliography}
\end{document}